\documentclass[10pt]{article}
\usepackage{mathrsfs}
\usepackage{amsmath,amssymb}
\usepackage{amsthm}
\usepackage[colorlinks,linkcolor=red,anchorcolor=blue,citecolor=green]{hyperref}
\usepackage{graphics,graphicx}
\usepackage{color}
\usepackage{enumerate}
 \usepackage[titletoc]{appendix}
\allowdisplaybreaks

\usepackage{tikz}
\usepackage{indentfirst}

\newtheorem{thm}{Theorem}[section]
\newtheorem{lem}[thm]{Lemma}

\newtheorem{cor}[thm]{Corollary}

\newtheorem{remark}[thm]{Remark}

\def\qed{\hfill \rule{4pt}{7pt}}

\def\pf{\noindent {\it{Proof.}\hskip 2pt}}

\def\R{{\mathbb{R}}}
\def\E{{\mathbb{E}}}
\def\N{{\mathbb{N}}}
\def\P{{\mathbb{P}}}

\def\Z{{\mathbb{Z}}}

\def\sO{{\mathcal{O}}}

\def\sX{{\mathcal{X}}}

\def\RW{{\mathrm{RW}}}
\def\zero{{\mathbf{0}}}

\begin{document}
\begin{center}
{{\large\bf The invariance principle and the large deviation for the biased random walk on $\mathbb{Z}^d$}\footnote{The project is supported partially by CNNSF (No.~11671216).}}
\end{center}
\vskip 2mm
\begin{center}
Yuelin Liu$^a$, Vladas Sidoravicius$^b$, Longmin Wang$^a$, Kainan Xiang$^a$
\vskip 1mm
\footnotesize{$^a$School of Mathematical Sciences, LPMC, Nankai University}\\
\footnotesize{Tianjin City, 300071, P. R. China}\\
\footnotesize{$^b$Courant Institute of Mathematical Sciences, New York, NY 10012, USA}\\
\footnotesize{\& NYU-ECNU Institute of Mathematical Sciences at NYU Shanghai}\\
\footnotesize{Emails: yuelinliu@mail.nankai.edu.cn} (Liu)\\
\footnotesize{~~~~~vs1138@nyu.edu} (Sidoravicius)\\
\footnotesize{~~~~~~wanglm@nankai.edu.cn} (Wang)\\
\footnotesize{~~~~~~~~~~~~kainanxiang@nankai.edu.cn} (Xiang)
\end{center}
\begin{abstract}
In this paper, we establish the invariance principle and the large deviation for the biased random walk $\RW_\lambda$ with $\lambda\in [0,1)$ on $\mathbb{Z}^d, d\geq 1.$\\


\end{abstract}

\section{Introduction}
\setcounter{equation}{0}
\noindent The biased random walk $\RW_\lambda$ with parameter $\lambda\in [0,\infty)$ was introduced to design a  Monte-Carlo algorithm for the self-avoiding walk by Berretti and Sokal \cite{BS1985}. The idea was refined and developed in \cite{LS1988,SJ1989, RD1994}.  Lyons,  and Lyons, Pemantle and Peres  produced a sequence of remarkable papers on  $\RW_\lambda$ (\cite{LR1990, LR1992, LR1995, LPP1996a, LPP1996b}).
$\RW_\lambda$  was studied in a number of works later (see for example \cite{BHOZ2013, AE2014, BFS2014, HS2015, BF2014} and the references therein). \cite{SSSWX2018} gave a spectral radius and several additional properties for biased random walk on  infinite graphs. In \cite{SSSWX2018b}, it was shown that there is a phase transition in relation of the tree number in the uniform spanning forest on Euclidean lattice equipped with a network corresponding to biased random walk.  Bowditch \cite{BA2018} proved a quenched invariance principle for $\{\vert X_n\vert\}_{n=0}^{\infty}$ when $\{X_n\}_{n=0}^{\infty}$ is a $\RW_\lambda$ on supercritical Galton-Watson tree, and showed that the corresponding scaling limit is a one dimensional Brownian motion. 


Our paper aims to study the invariance principle (IP) and the large deviation principle (LDP) for $\RW_\lambda $ with $\lambda\in [0,1)$ on $d$-dimensional integer lattice $\mathbb{Z}^d\ (d\geq 1).$ Main results of  this paper are Theorems \ref{thm2.1} and \ref{thm3.1}.
We define $\RW_\lambda\ (\lambda\geq 0)$ on $\mathbb{Z}^d$ as follows: Let $\mathbf{0}=(0,\cdots,0)\in \mathbb{Z}^d$ and
$$\vert x\vert=\sum\limits_{i=1}^d\vert x_i\vert,\ x=(x_1,\cdots,x_d)\in\mathbb{Z}^d$$
which is the graph distance between $x$ and $\mathbf{0}.$ Write $\mathbb{N}$ for the set of natural numbers, and let $\mathbb{Z}_{+}=\mathbb{N}\cup\{0\}.$ For any $n\in\mathbb{Z}_{+},$ define
$$B(n)=\left\{x\in \mathbb{Z}^d:\ \vert x\vert\leq n\right\},\ \partial B(n)=\left\{x\in \mathbb{Z}^d:\ \vert x\vert=n\right\}.$$
If the edge $e=\{x,y\}$ is at  graph distance $n$ from $\mathbf{0},$ namely $\vert x\vert\wedge\vert y\vert=n,$ then let its conductance to be $\lambda^{-n}.$ Denote by
$\RW_\lambda\ (X_n)_{n=0}^\infty$ the random walk associated to the above conductances and call it the biased random walk with parameter $\lambda.$ $\RW_{\lambda}$ $(X_n)_{n=0}^{\infty}$ has the following transition probability:
\begin{eqnarray}\label{(1.1)}
p(v,u):=p_{\lambda}(v,u)=\left\{\begin{array}{cl}
1/d_v &{\rm if}\ v=\mathbf{0},\\
\ \\
      \frac{\lambda}{d_v+\left(\lambda-1\right)d_v^-} &{\rm if}\ u\in \partial B(|v|-1)\ \mbox{and}\ v\neq \mathbf{0},\\
      \ \\
     \frac{1}{d_v+\left(\lambda-1\right)d_v^-} &{\rm otherwise}.
   \end{array}
\right.
\end{eqnarray}
Here $d_v=2d$ is the degree of the vertex $v;$ and $d_v^-$ (resp. $d_v^+$) is the number of edges connecting $v$ to $\partial B(|v|-1)$ (resp. $\partial B(|v|+1)$). Note that
$\RW_1$ $(X_n)_{n=0}^{\infty}$  is the simple random walk (SRW) on $\mathbb{Z}^d;$ and
\begin{eqnarray*}
d_v^+=d+\kappa (v),\ d_v^-=d-\kappa (v),\ v=(v_1,\cdots,v_d)\in\mathbb{Z}^d,
\end{eqnarray*}
where $\kappa(v)=\#\{i:\ v_i=0\}$ (with $\# A$ being the cardinality of a set $A$). For any $n\in\mathbb{Z}_+,$ write $X_n=\left(X_n^1,\cdots,X_n^d\right).$

It is known that on $\Z^d\ (d\geq1)$, $\RW_{\lambda}$ $(X_n)_{n=0}^{\infty}$ is transient for $\lambda<1$ and positive recurrent for $\lambda>1$ (R. Lyons \cite{LR1995}, R. Lyons and Y. Peres \cite[Theorem~3.10]{LP2017}).
Let
\[\sX = \left\{ x=(x_1,\cdots,x_d) \in \Z^d:\ x_i = 0 \text{ for some } 1 \leq i \leq d \right\}. \]
Then from \cite[Lemma 2.3 and Theorem 2.4]{SSSWX2018b},  the next two results were stated for $\lambda\in (0,1)$. In fact, for any $0\leq \lambda<1,$ with probability $1$, 
\begin{eqnarray}\label{facts}
\left\{\begin{array}{ll}
\RW_{\lambda}\ (X_n)_{n=0}^{\infty}\ \mbox{visits}\ \sX \ \mbox{only finitely many times;}\\
\ \ \ \\
\frac{1}{n} \left(\left|X_n^1\right|, \cdots, \left|X_n^d\right|\right) \rightarrow \frac{1- \lambda}{1+ \lambda} \left(\frac{1}{d},\cdots,\frac{1}{d}\right)\ \mbox{as}\ n\rightarrow\infty.
\end{array}
\right.
\end{eqnarray}
And when $\lambda>1,$ $(X_n)_{n=0}^{\infty}$ is positive recurrent and ergodic with ${\bf 0}$ speed: 
\[\lim\limits_{n\rightarrow\infty}\frac{\left(\left|X_n^1\right|, \cdots, \left|X_n^d\right|\right)}{n}=\mathbf{0}\]  almost surely. The related
central limit theorem (CLT) and IP can be derived from \cite{JG2004} and \cite{MPU2006} straightforwardly. 

Thus a natural question is to study CLT, IP with $\lambda\in [0,1)$ and LDP for $\RW_\lambda$ on $\mathbb{Z}^d.$ The LDP for $\RW_\lambda$ $(X_n)_{n=0}^{\infty}$ means LDP for the scaled reflected biased random walk $\left\{\frac{1}{n}\left(\left\vert X_n^1\right\vert,\cdots,\left\vert X_n^d\right\vert\right)\right\}_{n=1}^{\infty}$. And CLT, IP and LDP for $\RW_0$ are interesting
only for $d\geq 2.$  In this paper, Theorem \ref{thm2.1} proves CLT and IP for $\RW_\lambda, \lambda\in[0,1)$. However  the proof of the Theorem \ref{thm2.1}  shows that the  scaling limit in this case is not a $d$-dimensional Brownian motion. For Theorem \ref{thm3.1}, we derive LDP for scaled reflected $\RW_\lambda, \lambda\in[0,1)$. The rate function of scaled reflected $RW_{\lambda}$ differs from that of drifted random walk, though there is a strong connection between them. Here, we only can calculate the rate function $\Lambda^\ast$ out when $d\in\{1,2\}$ and $\lambda\in (0,1)$, $d\geq 2$ and $\lambda=0$. To obtain an explicit rate function when $d\geq 3$ and $\lambda\in (0,1)$ still remains as an open problem. And in the recurrent case when $\lambda>1$, we hope to establish a LDP with proper normalization in the future work.

\section{CLT and Invariance Principle for $\RW_\lambda$ with $\lambda\in [0,1)$}\label{sec2}
\setcounter{equation}{0}
\noindent In this section, we fix $\lambda\in [0,1),$ then use the matingale's CLT for $\mathbb{R}^d$-valued martingales and the martingale characterization of Markov chains
to prove the IP for reflected $\RW_\lambda$
(Theorem \ref{thm2.1}). And CLT for reflected $\RW_\lambda$ is a consequence of corresponding IP.

To describe our main result, we need to introduce some notations. For any nonnegative definite symmetric $d\times d$ matrix $A$, let $\mathcal{N}(\mathbf{0},A)$ be the normal distribution with mean $\mathbf{0}$ and covariance matrix $A.$ Define the following positive definite symmetric $d\times d$ matrix $\Sigma=(\Sigma_{ij})_{1\leq i,j\leq d}:$
$$\Sigma_{ii}=\frac{1}{d}-\frac{(1-\lambda)^2}{d^2(1+\lambda)^2},\ \Sigma_{ij}=-\frac{(1-\lambda)^2}{d^2(1+\lambda)^2},\ 1\leq i\not=j\leq d.$$
For a random sequence $\{Y_n\}_{n=1}^{\infty}\subset\mathbb{R}^d,$ $Y_n\rightarrow \mathcal{N}(\mathbf{0},A)$ means that as $n\rightarrow\infty,$ $Y_n$ converges in distribution to the normal distribution $\mathcal{N}(\mathbf{0},A)$ in Skorohod space $D\left([0,\infty),\mathbb{R}^d\right)$. 
For any $a\in\mathbb{R},$ let $\lfloor a\rfloor$ be the integer part of $a.$ Put
$$v=\left(\frac{1-\lambda}{d(1+\lambda)},\cdots,\frac{1-\lambda}{d(1+\lambda)}\right)\in\mathbb{R}^d.$$
\vskip 3mm

\begin{thm}[IP and CLT]\label{thm2.1} 
Let $0 \leq \lambda < 1$ and $(X_m)_{m=0}^{\infty}$ be $\RW_{\lambda}$ on $\mathbb{Z}^d$ starting at $x$, $x\in \mathbb{Z}^d$. 
Then, on $D\left([0,\infty),\mathbb{R}^d\right),$ 
\[\left(\frac{\left(\left|X_{\lfloor nt\rfloor}^1\right|,\cdots, \left|X_{\lfloor nt\rfloor}^d\right|\right)-nvt}{\sqrt{n}}\right)_{t\geq 0}\ 
   \]
converges in distribution to $\frac{1}{\sqrt{d}}\left[I- \frac{1-\rho_\lambda}{d}E\right]W_t$ as $n\rightarrow\infty$, where $(W_t)_{t\geq 0}$ is the $d$-dimensional Brownian motion starting at $\mathbf{0}$, $I$ is identity matrix and $E$ denotes the $d\times d$ matrix whose entries are all equal to $1$ and $\rho_\lambda = 2 \sqrt{\lambda}/(1+\lambda)$ is the spectral radius of $X_n$ (see \cite[Theorem~1.1]{SSSWX2018b}).

In particular, we have 
\begin{eqnarray*}
\frac{(|X_n^1|,\cdots,|X_n^d|)-nv}{\sqrt{n}}\rightarrow \mathcal{N}(\mathbf{0},\Sigma)\ 
\end{eqnarray*}
in distribution as $n \to \infty$. 
\end{thm}
\vskip 2mm

\begin{remark} Note $\Sigma=0$ when $d=1$ and $\lambda=0.$ Hence for $d=1$ and $\lambda=0,$ both
$$\frac{|X_n|-nv}{\sqrt{n}}=\frac{\vert X_0\vert+n-nv}{\sqrt{n}}\rightarrow \mathcal{N}(\mathbf{0},\Sigma)$$
and $\left(\frac{\left|X_{\lfloor nt\rfloor}\right|-nvt}{\sqrt{n}}\right)_{t\geq 0}$ converging in distribution to $Y=(Y_t)_{t\geq 0}$
are not interesting.

From Theorem \ref{thm2.1}, the following holds: For $\RW_{\lambda}$ $(X_m)_{m=0}^{\infty}$ on $\mathbb{Z}^d$ starting at any fixed vertex with $\lambda<1,$
on $D\left([0,\infty),\mathbb{R}\right),$ as $n\rightarrow\infty,$ $\left(\frac{\left|X_{\lfloor nt\rfloor}\right|-n\frac{1-\lambda}{1+\lambda}t}{\sqrt{n}}\right)_{t\geq 0}$
converges in distribution to $\left(\rho_{\lambda} B_t\right)_{t\geq 0}$ and  $(B_t)_{t\geq 0}$ is the $1$-dimensional Brownian motion starting at $0.$
\end{remark}
\vskip 2mm

\noindent{\bf Proof of Theorem \ref{thm2.1}.}
We only prove IP for $\left(\frac{\left(\left|X_{\lfloor nt\rfloor}^1\right|,\cdots, \left|X_{\lfloor nt\rfloor}^d\right|\right)-nvt}{\sqrt{n}}\right)_{t\geq 0}.$
Let
$$\mathcal{F}_k=\sigma(X_0,X_1,\cdots,X_k),\ k\in\mathbb{Z}_{+}.$$
Give any $1\leq i\leq d,$ and define $f_{i}: ~\mathbb{Z}^d\rightarrow \mathbb{R}$ as follows: for any $x=(x_1,\cdots,x_d)\in\mathbb{Z}^d,$
\begin{eqnarray}\label{generator}
f_{i}(x) = \left\{\begin{array}{cl}
      \frac{2}{d+\kappa_x+\lambda \left(d-\kappa_x\right)} &{\rm if} \ x_i= 0 ,\\
      \ \\
       \frac{1- \lambda}{d+\kappa_x+\lambda \left(d-\kappa_x\right)} &{\rm if} \ x_i\neq  0 ,\\
     \end{array}
\right.
\end{eqnarray}
then for any $k\in\mathbb{N},$
$$f_i\left(X_{k-1}^i\right)=\mathbb{E}\left(\left.\left|X_{k}^i\right|-\left|X_{k-1}^i\right|\ \right| \mathcal{F}_{k-1}\right).$$
By martingale characterization theorem of Markov chain $\{|X_k^i|\}_k$, \[\left\{\left|X_k^i\right|-\left|X_{k-1}^i\right|-f_i(X_{k-1})\right\}_{k\geq 1}\] is an $\mathcal{F}_k$-adapted martingale-difference sequence, and so is
$$\left\{\xi_{n, k}^i:= \frac{1}{\sqrt{n}}\left(\left|X_k^i\right|-\left|X_{k-1}^i\right|-f_i(X_{k-1})\right)\right\}_{k= 1}\ \mbox{for any}\ n\in\mathbb{N}.$$
For each random sequence $\left\{\xi_{n,k}:=\left(\xi_{n,k}^1,\cdots,\xi_{n,k}^d\right)\right\}_{k\geq 1},$ define $M^n=(M^n_t)_{t\geq 0}$ and $A_n=(A_n(t))_{t\geq 0}$ as follows
\begin{eqnarray*}
M_t^n=\left(M_t^{n,1},\cdots,M_t^{n,d}\right)=\sum\limits_{k= 1}^{\lfloor nt\rfloor} \xi_{n,k}, \ A_n(t)=\left(A_n^{i,j}(t)\right)_{1\leq i,j\leq d}=\sum\limits_{k= 1}^{\lfloor nt\rfloor}\left(\xi_{n,k}^i\xi_{n,k}^j\right)_{1\leq i,j\leq d},\ t\geq 0.
\end{eqnarray*}
Then each $\left(M_t^{n,i}\right)_{t\geq 0}$ is an $\mathcal{F}_{\lfloor nt\rfloor}$-martingale with $M^{n,i}_0=0,$
and further each $M^n$ is a $\mathbb{R}^d$-valued $\mathcal{F}_{\lfloor nt\rfloor}$-martingale with $M_0=\mathbf{0}.$


Note that for any $t>0,$
$$\left\vert M_t^n-M^n_{t-}\right\vert\leq \left\vert \xi_{n,\lfloor nt\rfloor}\right\vert =\sum\limits_{i=1}^d\left\vert\xi_{n,\lfloor nt\rfloor}^i\right\vert\leq \frac{2d}{\sqrt{n}},$$
and for any $T\in (0,\infty),$
\begin{eqnarray}\label{bound}
\lim\limits_{n\rightarrow\infty}\mathbb{E}\left[\sup_{t\leq T}\left|M^n_t-M^n_{t-}\right|\right]\leq\lim\limits_{n \rightarrow \infty} \frac{2d}{\sqrt n} = 0.\end{eqnarray}

Fix any $1\leq i,j\leq d$ and $0\leq s<t<\infty.$ Then by the martingale property,
\begin{eqnarray*}
\mathbb{E}\left[\left. M_s^{n,j}\left(M_t^{n,i}-M_s^{n,i}\right)\right\vert \mathcal{F}_{\lfloor ns\rfloor}\right]=M_s^{n,j}
  \mathbb{E}\left[\left. M_t^{n,i}-M_s^{n,i}\right\vert \mathcal{F}_{\lfloor ns\rfloor}\right]=0,
\end{eqnarray*}
and similarly $\mathbb{E}\left[\left. M_s^{n,i}\left(M_t^{n,j}-M_s^{n,j}\right)\right\vert \mathcal{F}_{\lfloor ns\rfloor}\right]=0.$ Additionally
\begin{eqnarray*}
&&\mathbb{E}\left[\left.\left(M_t^{n,i}-M_s^{n,i}\right)\left(M_t^{n,j}-M_s^{n,j}\right)\right\vert\mathcal{F}_{\lfloor ns\rfloor}\right]\\
&&=\mathbb{E}\left[\left.\sum\limits_{k=\lfloor ns\rfloor+1}^{\lfloor nt\rfloor}\xi_{n,k}^i\xi_{n,k}^j\right\vert\mathcal{F}_{\lfloor ns\rfloor}\right]+
    \mathbb{E}\left[\left.\sum\limits_{\lfloor ns\rfloor+1\leq r<k\leq \lfloor nt\rfloor}\xi_{n,k}^i\xi_{n,r}^j\right\vert\mathcal{F}_{\lfloor ns\rfloor}\right]\\
&&\ \ \ \ \ + \mathbb{E}\left[\left.\sum\limits_{\lfloor ns\rfloor+1\leq r<k\leq \lfloor nt\rfloor}\xi_{n,k}^j\xi_{n,r}^i\right\vert\mathcal{F}_{\lfloor ns\rfloor}\right]\\
&&=\mathbb{E}\left[\left.\sum\limits_{k=\lfloor ns\rfloor+1}^{\lfloor nt\rfloor}\xi_{n,k}^i\xi_{n,k}^j\right\vert\mathcal{F}_{\lfloor ns\rfloor}\right]+
    \sum\limits_{\lfloor ns\rfloor+1\leq r<k\leq \lfloor nt\rfloor}\mathbb{E}\left[\left.\mathbb{E}\left[\left.\xi_{n,k}^i\xi_{n,r}^j\right\vert\mathcal{F}_r\right]
    \right\vert\mathcal{F}_{\lfloor ns\rfloor}\right]\\
&&\ \ \ \ \ + \sum\limits_{\lfloor ns\rfloor+1\leq r<k\leq \lfloor nt\rfloor}\mathbb{E}\left[\left.\mathbb{E}\left[\left.\xi_{n,k}^j\xi_{n,r}^i\right\vert\mathcal{F}_r\right]
    \right\vert\mathcal{F}_{\lfloor ns\rfloor}\right]\\
&&=\mathbb{E}\left[\left.\sum\limits_{k=\lfloor ns\rfloor+1}^{\lfloor nt\rfloor}\xi_{n,k}^i\xi_{n,k}^j\right\vert\mathcal{F}_{\lfloor ns\rfloor}\right]=
 \mathbb{E}\left[\left. A_n^{i,j}(t)-A_n^{i,j}(s)\right\vert\mathcal{F}_{\lfloor ns\rfloor}\right].
\end{eqnarray*}
Since
\begin{eqnarray*}
&&M_t^{n,i}M_t^{n,j}-A_n^{i,j}(t)\\
&&=M_s^{n,i}M_s^{n,j}-A_n^{i,j}(s)+M_s^{n,i}\left(M_t^{n,j}-M_s^{n,j}\right)+M_s^{n,j}\left(M_t^{n,i}-M_s^{n,i}\right)\\
&&\ \ \ +\left(M_t^{n,i}-M_s^{n,i}\right)\left(M_t^{n,j}-M_s^{n,j}\right)-\left(A_n^{i,j}(t)-A_n^{i,j}(s)\right),
\end{eqnarray*}
we see that
\begin{eqnarray}\label{martingale}
\mathbb{E}\left[\left. M_t^{n,i}M_t^{n,j}-A_n^{i,j}(t)\right\vert \mathcal{F}_{\lfloor ns\rfloor}\right]=M_s^{n,i}M_s^{n,j}-A_n^{i,j}(s),
\end{eqnarray}
which implies that $A_n^{i,j}= \left[M_n^i, M_n^j\right]$ with $[X, Y]$ being the cross-variation process of $X$ and $Y$. 

For any $1\leq i\neq j\leq d,$ by (\ref{facts}) and (\ref{generator}), almost surely,
\begin{eqnarray*}
f_i\left(X_{k-1}\right)=f_j\left(X_{k-1}\right)=\frac{1-\lambda}{d(1+\lambda)}\ \mbox{for large enough}\ k,
\end{eqnarray*}
and as $n\rightarrow\infty,$
\begin{eqnarray*}
&&\frac{1}{n} \sum\limits_{k=1}^{\lfloor nt \rfloor}\left(\left|X_{k}^j\right|-\left|X_{k-1}^j\right|\right)=\frac{1}{n}\left(\left|X_{\lfloor nt \rfloor}^j\right|-\left|X_{0}^j\right|\right)\rightarrow
  \frac{(1-\lambda)t}{d(1+\lambda)}\\
 \end{eqnarray*}
which also holds with $j$ replaced by $i$. Thus
\[
\begin{split}
&\sum_{k=1}^{\lfloor nt \rfloor} \xi_{n,k}^{i}\xi_{n, k}^{j}= \sum_{k=1}^{\lfloor nt \rfloor} \frac{(|X_{k}^i|-|X_{k-1}^i|-f_{i}(X_{k-1}))}{\sqrt{n}} \frac{(|X_{k}^j|-|X_{k-1}^j|-f_{j}(X_{k-1}))}{\sqrt{n}} \\
&\ \ \ \ =\frac{1}{n} \sum\limits_{k=1}^{\lfloor nt \rfloor} \left\{\left(|X_{k}^i|-|X_{k-1}^i|\right)(|X_{k}^j|-|X_{k-1}^j|)-
       (|X_{k}^i|-|X_{k-1}^i|) f_j(X_{k-1})\right.\\
&\ \ \ \ \ \ \ \ \ \ \ \ \ \ \ \ \ \ -\left. (|X_{k}^j|-|X_{k-1}^j|) f_i(X_{k-1})+ f_i(X_{k-1}) f_j(X_{k-1})\right\}\\
&\ \ \ \ =\frac{1}{n} \sum\limits_{k=1}^{\lfloor nt \rfloor} \left\{-
       (|X_{k}^i|-|X_{k-1}^i|) f_j(X_{k-1}) -(|X_{k}^j|-|X_{k-1}^j|) f_i(X_{k-1})+ f_i(X_{k-1}) f_j(X_{k-1})\right\}\\
&\ \ \ \ \rightarrow -\frac{(1-\lambda)^2}{d^2 (1+\lambda)^2}t,
\end{split}
\]
where we use the fact that $(|X_{k}^i|-|X_{k-1}^i|)(|X_{k}^j|-|X_{k-1}^j|)=0,\ k\geq 1.$

On the other hand, for any fixed $1\leq i\leq d,$ construct the following martingale-difference sequence $\left(\zeta_{n, k}^i\right)_{k\geq 1}:$
$$\zeta_{n, k}^i=\left(\sqrt{n}\xi_{n, k}^i\right)^2- \mathbb{E}\left[\left.\left(\sqrt{n} \xi_{n, k}^i\right)^2\right| \mathcal{F}_{k-1}\right],\ k\in\mathbb{N}.$$
Then for any $1\leq k<\ell<\infty,$ $\mathbb{E}\left[\zeta_{n,k}^i\right]=\mathbb{E}\left[\zeta_{n,\ell}^i\right]=0,$ and
$$\mathbb{E}\left[\zeta_{n, k}^i\zeta_{n, \ell}^i\right]=\mathbb{E}\left[\zeta_{n,k}^i\mathbb{E}\left[\left.\zeta_{n, \ell}^i\right\vert \mathcal{F}_{\ell-1}\right]\right]=0,$$
which implies $\left(\zeta_{n, k}^i\right)_k$ is a sequence of uncorrelated random variables. Notice that for any $k\in\mathbb{N},$
$$\left\vert \zeta_{n,k}^i\right\vert\leq 2\ \mbox{and hence}\ {\rm Var}\left(\zeta_{n,k}^i\right)\leq 4.$$
By the strong law of large numbers for uncorrelated random variables (\cite[Theorem~13.1]{LP2017}), we have that almost surely, as $n\rightarrow\infty,$
\begin{eqnarray}\label{lln}
\frac{1}{n}\sum\limits_{k=1}^{\lfloor nt\rfloor}\zeta^i_{n,k}=\sum\limits_{k=1}^{\lfloor nt\rfloor}\left\{
    \left(\xi_{n, k}^i\right)^2- \mathbb{E}\left[\left.\left(\xi_{n,k}^i\right)^2\right| \mathcal{F}_{k-1}\right]\right\}\rightarrow 0.
\end{eqnarray}
Due to (\ref{facts}) and (\ref{generator}),
almost surely, as $n\rightarrow\infty,$
\begin{eqnarray*}
&&\sum\limits_{k=1}^{\lfloor nt \rfloor}\mathbb{E}\left[\left.\left(\xi_{n, k}^i\right)^2\right| \mathcal{F}_{k-1}\right]= \sum\limits_{k=1}^{\lfloor nt \rfloor}
  \mathbb{E}\left[\left.\left(\frac{\left|X_{k}^i\right|-\left|X_{k-1}^i\right|- f_{i}(X_{k-1})}{\sqrt{n}}\right)^2\right| \mathcal{F}_{k-1}\right] \\
&&=\frac{1}{n}\sum\limits_{k=1}^{\lfloor nt \rfloor}\left\{\mathbb{E}\left[\left.\left(\left|X_{k}^i\right|-\left|X_{k-1}^i\right|\right)^2\right| \mathcal{F}_{k-1}\right]-
               2\mathbb{E}\left[\left.\left(\left|X_{k}^i\right|-\left|X_{k-1}^i\right|\right)f_i(X_{k-1}^i)\right| \mathcal{F}_{k-1}\right]\right.\\
&& \ \ \ \ \ \ \ \ \ \ \ \ \ \ \ \left.+ \mathbb{E}\left[\left.\left(f_i\left(X_{k-1}^i\right)\right)^2\right| \mathcal{F}_{k-1}\right]\right\}\\
&&=\frac{1}{n} \sum\limits_{k=1}^{\lfloor nt \rfloor}\left\{\mathbb{E}\left[\left.\left(\left|X_{k}^i\right|-\left|X_{k-1}^i\right|\right)^2\right| \mathcal{F}_{k-1}\right]-\left(f_i\left(X_{k-1}^i\right)\right)^2\right\}\\
&&\rightarrow t \frac{1}{d }-t \frac{(1-\lambda)^2}{d^2 (1+\lambda)^2}.
\end{eqnarray*}
Together with (\ref{lln}), we have 
\begin{eqnarray}\label{crossvariation}
\lim\limits_{n\rightarrow \infty}A_n(t)=t\Sigma.
\end{eqnarray}
Therefore (\ref{bound}), (\ref{martingale}) and (\ref{crossvariation}) implies that (\cite[Theorem~1.4]{EK2005})  on $D\left([0,\infty),\mathbb{R}^d\right),$ as $n\rightarrow\infty,$
\[\left(\frac{\left(\left|X_{\lfloor nt\rfloor}^1\right|,\cdots, \left|X_{\lfloor nt\rfloor}^d\right|\right)-nvt}{\sqrt{n}}\right)_{t\geq 0}\]
converges in distribution to a $\mathbb{R}^d$-valued process $Y :Y_t=\frac{1}{\sqrt{d}}\left[I- \frac{1-\rho_\lambda}{d}E\right]W_t$ with independent Gaussian increments such that $Y_0=\mathbf{0}$ and $Y_{t+s}-Y_s$ has the law $\mathcal{N}(\mathbf{0},t\Sigma)$
for any $0\leq s, t<\infty.$ \qed

\section{LDP for scaled reflected ${\rm RW}_\lambda$ with $\lambda\in [0,1)$}\label{sec3}
\setcounter{equation}{0}
\noindent In this section, we fix $\lambda\in [0,1)$ when $d\geq 2$ and $\lambda\in (0,1)$ when $d=1,$
then prove that the sequence $\left\{\frac{1}{n}\left(\left|X_n^1\right|, \cdots ,\left|X_n^d\right|\right)\right\}_{n=1}^{\infty}$
satisfy the LDP with a good rate function as $n\rightarrow\infty$ (Theorem \ref{thm3.1}).

Write $\mathbb{R}_+=[0,\infty).$ Let
$$s_0:=s_0(\lambda)=\frac{1}{2}\ln\lambda,\ \rho_\lambda=\frac{2\sqrt{\lambda}}{1+\lambda}.$$
Here $\rho_\lambda$ is the spectral radius of $\RW_\lambda$ as in Theorem \ref{thm2.1} and by \cite[Theorem~2.1]{SSSWX2018b}, for any $\lambda\in [0,1)$,
\begin{eqnarray*}
p_{\lambda}^{(2n)} (\mathbf{0},\mathbf{0})=\mathbb{P}\left[X_{2n}=\mathbf{0}\vert X_0=\mathbf{0}\right]\asymp \rho_\lambda^{2n} \frac{1}{n^{3d/2}}.
\end{eqnarray*}
Here for two nonnegative sequences $\{a_n\}_{n=1}^{\infty}$ and $\{b_n\}_{n=1}^{\infty},$ $a_n\asymp b_n$ means that there are two positive constants $c_1$ and $c_2$ such that
$c_1b_n\leq a_n\leq c_2b_n$ for large enough $n.$ For any $s=(s_1, \cdots, s_d)\in\mathbb{R}^d,$ let
\begin{eqnarray}\label{psi}
\psi(s)=N(s)\frac{\rho _\lambda}{d}+\frac{1}{d(1+ \lambda)}\sum\limits_{i=1}^d\left\{\lambda e^{-s_i}+ e^{s_i}\right\}I_{\{s_i\geq s_0\}},\ \mbox{where}\ N(s)=\sum\limits_{i=1}^dI_{\{s_i<s_0\}}.
\end{eqnarray}

\begin{thm}[LDP]\label{thm3.1} Fix $\lambda\in [0,1)$ when $d\geq 2,$ and $\lambda\in (0,1)$ when $d=1.$ Assume $\RW_\lambda$ $\{X_n\}_{n=1}^{\infty}$ starts at any fixed point in $\mathbb{Z}^d.$

{\bf (i)} For any $d\in\mathbb{N},$
$\left\{\frac{1}{n}\left(\left|X_n^1\right|, \cdots ,\left|X_n^d\right|\right)\right\}_{n=1}^{\infty}$ satisfies the LDP with the following good rate function:
\begin{eqnarray*}
\Lambda^{\ast}(x)= \sup_{s\in \mathbb{R}^d}\left\{(s,x)-\ln\psi (s)\right\},\ x\in\mathbb{R}^d.
\end{eqnarray*}
In addition, when $\lambda\in (0,1),$
$$\mathcal{D}_{\Lambda^\ast}:=\left\{x\in\mathbb{R}^d:\ \Lambda^\ast(x)<\infty\right\}=\left\{x=(x_1,\cdots,x_d)\in\mathbb{R}^d_{+}:\ 0\leq \sum\limits_{i=1}^dx_i\leq 1\right\},$$
$$\left\{x\in\mathbb{R}^d:\ \Lambda^\ast(x)=0\right\}=\left\{\left(\frac{1-\lambda}{d(1+\lambda)},\cdots,\frac{1-\lambda}{d(1+\lambda)}\right)\right\};$$
and when $d\geq 2$ and $\lambda=0,$
$$\mathcal{D}_{\Lambda^\ast}(0):=\left\{x\in\mathbb{R}^d:\ \Lambda^\ast(x)<\infty\right\}=\left\{x=(x_1,\cdots,x_d)\in\mathbb{R}^d_{+}:\ \sum\limits_{i=1}^dx_i=1\right\},$$
$$\left\{x\in\mathbb{R}^d:\ \Lambda^\ast(x)=0\right\}=\left\{\left(\frac{1}{d},\cdots,\frac{1}{d}\right)\right\}.$$

{\bf (ii)} In particular, when $d=1, \lambda \in (0, 1),$
\begin{eqnarray*}
\Lambda^{\ast}(x)=\left\{\begin{array}{cl}
\frac{x}{2}\ln\lambda - \ln \rho_\lambda+ (1+x)\ln \sqrt{1+ x}+ (1-x)\ln\sqrt{1-x}, & x\in [0,1],\\
\ \\
  +\infty, & otherwise;
   \end{array}
\right.
\end{eqnarray*}
and when $d= 2, \lambda \in (0, 1),$
\begin{eqnarray*}
\Lambda^{\ast}(x)=\left\{\begin{array}{cl}
\frac{1}{2}(x_1+x_2)\ln\lambda-\ln\rho_\lambda +\overline{\Lambda^*}(x), & (x_1, x_2)\in\mathcal{D}_{\Lambda^*},\\
\ \\
  +\infty, & otherwise,
   \end{array}
\right.
\end{eqnarray*}
where for any $x=(x_1,x_2)\in\mathbb{R}^2_+$ with $x_1+x_2\leq 1,$
\begin{eqnarray*}
&&\overline{\Lambda^*}(x)=x_1\ln \left[\frac{2x_1+\sqrt{x_1^2-x_2^2+1}}{\sqrt{(x_1^2-x_2^2)^2+1-2(x_1^2+x_2^2)}}\right]+x_2\ln \left[\frac{2x_2+\sqrt{x_2^2-x_1^2+1}}{\sqrt{(x_1^2-x_2^2)^2+1-2(x_1^2+x_2^2)}}\right]\\
\ \\
&&\ \ \ \ \ \ \ \ \ \ \ - \ln \left[\frac{\sqrt{x_1^2-x_2^2+1}+\sqrt{x_2^2-x_1^2+1}}{\sqrt{(x_1^2-x_2^2)^2+1-2(x_1^2+x_2^2)}}\right]+\ln 2.
\end{eqnarray*}
In addition, when $d\geq 2, \lambda=0,$
\begin{eqnarray*}
&&\Lambda^\ast (x)=\left\{\begin{array}{cl}
  \ln d+\sum\limits_{i=1}^dx_i\ln x_i,\ &x\in\mathcal{D}_{\Lambda^\ast}(0),\\
  \ \\
  +\infty, \ & othwise.
  \end{array}
\right.
\end{eqnarray*}

\end{thm}
\vskip 2mm

\begin{remark}
{\bf (i)} When $d\geq 3$ and $\lambda\in (0,1),$ we can not calculate $\Lambda^\ast$ out to get an explicit expression.
But recall that the rate function of the Cram\'{e}r theorem for SRW is
$$\overline{\Lambda^*}(x)=\sup\limits_{y\in\mathbb{R}^d}\ln\left\{\frac{e^{\sum\limits_{i=1}^dy_ix_i}}{\frac{1}{2d}\sum\limits_{i=1}^d(e^{-y_i}+e^{y_i})}\right\},\ x\in\mathbb{R}^d.$$
By (\ref{rate1}), for any $x=(x_1,\cdots,x_d)\in\mathbb{R}_+^d,$
$$\Lambda^*(x)=\frac{1}{2}\sum\limits_{i=1}^dx_i\ln\lambda-\ln\rho_\lambda+\overline{\Lambda^*}(x).$$
To obtain an explicit rate function when $d\geq 3$ and $\lambda\in (0,1)$,  it remains as an open problem.

{\bf (ii)} Assume $\lambda\in [0,1)$ if $d\geq 2$ and $\lambda\in (0,1)$ if $d=1.$ The following sample path large deviation (Mogulskii type theorem) holds for the reflected $\RW_\lambda$ starting at $\mathbf{0}.$ For any $0\leq t\leq 1,$ let
\begin{eqnarray*}
&&Z_n(t)=\frac{1}{n}\left(\left\vert X_{\lfloor nt\rfloor}^1\right\vert,\cdots,\left\vert X_{\lfloor nt\rfloor}^d\right\vert\right),\\
&&\widetilde{Z}_n(t)=Z_n(t)+\left(t-\frac{\lfloor nt\rfloor}{n}\right)\left(\left(\left\vert X_{\lfloor nt\rfloor+1}^1\right\vert,\cdots,\left\vert X_{\lfloor nt\rfloor+1}^d\right\vert\right)-\left(\left\vert X_{\lfloor nt\rfloor}^1\right\vert,\cdots,\left\vert X_{\lfloor nt\rfloor}^d\right\vert\right)\right).
\end{eqnarray*}
Write $L_{\infty}([0,1])$ for the $\mathbb{R}^d$-valued $L_{\infty}$ space on interval $[0,1],$ and $\mu_n$ (resp. $\widetilde{\mu}_n$) the law of $Z_n(\cdot)$ (resp. $\widetilde{Z}_n(\cdot)$) in $L_{\infty}([0,1]).$ From \cite[Lemma~5.1.4]{DZ1998}, $\mu_n$ and $\tilde{\mu}_n$ are exponentially equivalent.
Denote by $\mathcal{AC}^{'}$ the space of non-negative absolutely continuous $\mathbb{R}^d$-valued functions $\phi$ on $[0,1]$ such that
$$||\phi||\leq 1,\ \phi(0)=\mathbf{0}$$
where $\Vert\cdot\Vert$ is the supremum norm. 
Let
$C_{\mathbf{0}}([0,1])=\left\{\left. f:\ [0,1]\rightarrow\mathbb{R}^d\right\vert f\ \mbox{is continuous},\ f(0)=\mathbf{0}\right\},$
and
$$\mathcal{K}=\left\{f\in C_{\mathbf{0}}([0,1])\left\vert \sup\limits_{0\leq s<t\leq 1}\frac{\left\vert f(t)-f(s)\right\vert}{t-s}\leq 1,\
   \sup\limits_{0\leq t\leq 1}\vert f(t)\vert\leq 1\right.\right\}.$$
By the Arzel\`{a}-Ascoli theorem, $\mathcal{K}$ is compact in $\left(C_{\mathbf{0}}([0,1]),\Vert\cdot\Vert\right)$.  Note that each $\widetilde{\mu}_n$ concentrates on $\mathcal{K}$ which implies exponential tightness in $C_0[0, 1]$ equipped with supremum topology. Given that the finite dimensional LDPs for $\mu_n\circ p_j^{-1}$ in $(\R^d)^{|j|}$ follows from Theorem \ref{thm3.1}, $\mu_n\circ p_j^{-1}$ and $\tilde{\mu}_n\circ p_j^{-1}$ are exponential equivalent, $\tilde{\mu}_n$ satisfy the LDP in $L_{\infty}[0,1]$ with pointwise convergence topology by Dawson-G\"{a}rtner theorem (\cite{DZ1998}). Together with exponential tightness in $C_0[0, 1]$, we can lift LDP of $\tilde{\mu}_n$ to $L_{\infty}[0,1]$ with supremum topology with the good rate function \[
I(\phi)=\left\{\begin{array}{cl}
  \int_0^1\Lambda^*\left(\dot{\phi}(t)\right)\ {\rm d}t &{\rm if}\ \phi\in\mathcal{AC}^{'},\\
  \ \\
     \infty  & {\rm otherwise}.
\end{array}
\right.
\]
The same holds for $\mu_n$ due to exponential equivalence. The proof of the above sample path LDP is similar to that of \cite[Theorem~5.1.2]{DZ1998}. 
\end{remark}
\vskip 2mm

For any $n\in\mathbb{N},$ $s=(s_1,\cdots,s_d)\in \mathbb{R}^d$ and $x\in\mathbb{Z}^d,$ let
\[
\Lambda_n (s,x)= \ln\mathbb{E}_x\left[\exp\left\{\left(n s, \left(\frac{|X^1_n|}{n},\cdots,\frac{\vert X^d_n\vert}{n}\right)\right)\right\}\right]
  =\ln\mathbb{E}_x\left[\exp\left\{\sum\limits_{i=1}^d s_i|X_n^i|\right\} \right].
\]

To prove Theorem \ref{thm3.1}, we firstly prove the following lemmas.

\begin{lem}
  \label{L:Lambda0}
  For each $s \in \R^d$ and $x \in \Z^d$, there exists a constant $C > 0$ such
  that
  \begin{equation}
    \label{e:Lambda0}
    \Lambda_{n-|x|}(s,\, x) - C \leq \Lambda_n(s,\, \zero) \leq \Lambda_{n+|x|}(s,\, x) + C, \quad n > |x|.  
  \end{equation}
\end{lem}

\begin{proof}
  By the Markov property, we have
  \[ \E_\zero \left[ e^{\sum_{i=1}^d s_i |X_n^i|} \right]
    \geq
    \E_0 \left[ I_{\{X_{|x|}=x\}} e^{\sum_{i=1}^d s_i |X_n^i|} \right]
     \geq
      \P_{\zero} \left( X_{|x|} = x \right) \E_x \left[ e^{\sum_{i=1}^d s_i
          |X_{n-|x|}^i|} \right],
    \]
  and the first inequality in \eqref{e:Lambda0} follows. The second inequality
  is proved similarly. 
\end{proof}

Let $e_1$, $\ldots$, $e_d$ be the standard unit vectors in $\mathbb{Z}^d$. Suppose
$\left\{Z_n=\left(Z_n^1,\,\cdots,\,Z_n^d\right)\right\}_{n=0}^{\infty}$ is a
drifted random walk on $\mathbb{Z}^d$ such that for any $1\leq i\leq d$ and
$n\in\mathbb{Z}_+$, 
\begin{equation}
  \label{e:Zn}
  \mathbb{P}\left (Z_{n+1}=Z_n+e_i \,\big|\, Z_0,\,\cdots,\, Z_n\right)=\frac{1}{d(1+ \lambda)},\quad
\mathbb{P}\left(Z_{n+1}=Z_n-e_i \,\big|\, Z_0,\,\cdots,\,Z_n\right)=\frac{\lambda}{d(1+ \lambda)}.
\end{equation}

\begin{lem}
  \label{L:hkupper}
  We have that for $k \in \Z_+^d$,
  \begin{equation*}
    \label{e:hkupper}
    \P_\zero \left( X_n = k \right) \leq \P \left( Z_n = k \,\big|\, Z_0 = \zero \right), \quad n \geq 0. 
  \end{equation*}
\end{lem}

\begin{proof}
  For $x$, $k \in \Z^d$ and $n \in \N$, let $\Gamma_n(x,\, k)$ be the set of all
  nearest-neighbor paths  in $\Z^d$ from $x$ to $k$ with length $n$. For a path $\gamma =
  \gamma_0\gamma_1\cdots\gamma_n \in \Gamma_n(x,\, k)$, let
  \[ p(\gamma) = p(x,\, \gamma_1) p(\gamma_1,\, \gamma_2) \cdots
    p(\gamma_{n-1},\, k). \]
  Consider the first $n$ steps of $\RW_{\lambda}$ along the path $\gamma \in
  \Gamma_n(0,\, k)$. Each
  time the transition probability for the walk is either $\frac{1}{d + m +
    (d-m)\lambda}$ (with $0 \leq m \leq d$) or $\frac{\lambda}{d+m +
    (d-m)\lambda}$ (with $0 \leq m \leq d-1$). The total number of probability
  terms of the forms $\frac{1}{d + m +
    (d-m)\lambda}$ (resp. $\frac{\lambda}{d + m +
    (d-m)\lambda}$) is exactly $\frac{n + |k|}{2}$ (resp. $\frac{n - |k|}{2}$).
  Note that $d+m + (d-m)\lambda \geq d(1+\lambda)$. Therefore, we have for
  $\gamma \in \Gamma_n(\zero,\, k)$,
  \[ p(\gamma) \leq \left( \frac{1}{d(1+\lambda)} \right)^{\frac{n+|k|}{2}}
    \left( \frac{\lambda}{d(1+\lambda)} \right)^{\frac{n-|k|}{2}}. \]
  As a consequence,
  \[ \P_{\zero}(X_n = k) = \sum_{\gamma \in \Gamma_n(\zero,\, k)} p(\gamma) \leq
    \sum_{\gamma \in \Gamma_n(\zero,\, k)} \left( \frac{1}{d(1+\lambda)} \right)^{\frac{n+|k|}{2}}
    \left( \frac{\lambda}{d(1+\lambda)} \right)^{\frac{n-|k|}{2}} = \P \left(
      Z_n = k \, \big|\, Z_0 = \zero \right). \]
\end{proof}

\begin{lem}
  \label{L:hklower}
  For every $k \in \Z_+^d$ and $z \in \Z_+^d \setminus \mathcal{X}$,
  we have that
  \begin{equation*}
    \label{e:hklower}
    \P_z \left( X_n = k \right) \geq n^{-d} \P \left( Z_n = k \,\big|\, Z_0 = z \right). 
  \end{equation*}
\end{lem}

\begin{proof}
  Recall that $\mathcal{X}$ is the boundary of $\Z_+^d$. Define
  \[ \sigma = \inf \left\{ n:\ X_n \in \mathcal{X} \right\}, \quad \tau = \inf
    \left\{ n:\ Z_n \in \mathcal{X} \right\}. \]
  Starting at $z \in \Z_+^d \setminus \mathcal{X}$, the process $\left( X_n
  \right)_{0 \leq n \leq \sigma}$ has the same distribution as the drifted
  random walk $(Z_n)_{0 \leq n \leq \tau}$. Then we have for $k \in \Z_+^d$, 
  \begin{equation}
    \label{e:XnZn}
    \P_z \left( X_n = k \right) \geq \P \left( Z_n = k,\, n \leq \tau
      \,\big|\, Z_0 = z \right).
  \end{equation}
  
  For $\alpha$, $\beta \in \Z$ and $n \in \N$, denote by $P_{n,\,
    \beta}(\alpha)$ the number of 
  paths $\gamma = \gamma_0 \gamma_1 \cdots \gamma_n$ in $\Z$ with $\gamma_0 = \alpha$
  and $\gamma_n = \beta$, and by $Q_{n,\, \beta}(\alpha)$ the number of those
  paths with 
  additional property that $\gamma_i \geq \alpha \wedge \beta$ for $0 \leq i
  \leq n$. By \cite[Theorem~4.3.2 and Lemma~4.3.3]{DR2010}, we have that
  \begin{equation}
    \label{e:PnQn}
    Q_{n,\, \beta}(\alpha) \geq \frac{|\alpha-\beta| \vee 1}{n} P_{n,\, \beta}(\alpha). 
  \end{equation}

  For $k = (k_1,\, \cdots,\, k_d) \in \Z_+^d$ and $z = (z_1,\, \cdots,\, z_d)
  \in \Z_+^d \setminus \sX$, write $a = \frac{n + |k| - |z|}{2}$ and $b =
  \frac{n - |k| + |z|}{2}$. By \eqref{e:PnQn}, we obtain that 
  \begin{align*}
    & \P \left( Z_n = k,\, n \leq \tau \vert Z_0= z\right) \\
    \geq &
           \sum_{m_1+\cdots+m_d = n} {n \choose m_1,\, \cdots,\, m_d} d^{-n} \left( \frac{\lambda}{1+\lambda} \right)^b \left( \frac{1}{1+\lambda} \right)^a \prod_{m_j > 0} Q_{m_j,\, k_j}(z_j) \\
    \geq &
           \sum_{m_1+\cdots +m_d = n} {n \choose m_1,\, \cdots,\, m_d} d^{-n} \left( \frac{\lambda}{1+\lambda} \right)^b \left( \frac{1}{1+\lambda} \right)^a \prod_{m_j > 0} \frac{|k_j-z_j|\vee 1}{m_j} P_{m_j,\, k_j}(z_j) \\
    \geq &
           n^{-d} \sum_{m_1+\cdots+m_d=n} {n \choose m_1,\, \cdots,\, m_d} d^{-n} \left( \frac{\lambda}{1+\lambda} \right)^b \left( \frac{1}{1+\lambda} \right)^a \prod_{m_j > 0} P_{m_j,\, k_j}(z_j) \\
    = &
        n^{-d} \P \left( Z_n = k \,\big|\, Z_0 = z \right), 
  \end{align*}
which together with \eqref{e:XnZn} proves this lemma. 
\end{proof}

Recall that $s_0 = \frac{1}{2} \ln \lambda$. 
For fixed $s \in \R^d$, let $I_1 = \left\{ 1 \leq i \leq d:\ s_i < s_0 \right\}$
and $I_2 = \left\{ 1,\, \cdots,\, d \right\} \setminus I_1$. Define $s^{I_1} =
\left( s_i \right)_{i \in I_1}$ and $s^{I_2}$, $Z_n^{I_1}$, $Z_n^{I_2}$ are
defined similarly.

\begin{lem}
  \label{L:orthant}
  Let $y_0 = \zero^{I_1}$ if $n$ is even and $y_0 = (1,\, 0,\, \cdots,\, 0) \in \R^{N(s)}$
  otherwise. 
  For any $z \in \Z^d$ and $s \in
  \R^d$, 
  we have that 
  \begin{equation*}
    \label{e:orthant}
    \E \left[ e^{\sum_{i=1}^d s_i Z_n^i} I_{\{Z_n^i \geq z_i\}} \,\big|\, Z_0 = z \right]
    \geq
    2^{-|I_2|} \left( e^{s_1} \wedge 1 \right) \E \left[ e^{\sum_{i\in I_2} s_i Z_n^i} I_{\{Z_n^{I_1} = z^{I_1} + y_0\}} \,\big|\, Z_0 = z \right].  
  \end{equation*}
\end{lem}

\begin{proof}
  Without loss of generality, we assume that $z = \zero$. 
  As in the proof of Lemma~\ref{L:hklower}, for $\alpha$, $\beta \in \Z$, we
  denote by $P_{n,\,\beta}(\alpha)$ the number of nearest-neighbor paths in $\Z$
  from $\alpha$ to $\beta$.  Then we have, for $k \in \Z^d$,
  \begin{align}\label{orthant1}
    \nonumber& \P \left( Z_n = k \,\big|\, Z_0 = \zero \right) \\
    =& \sum_m {n \choose m_1,\, \cdots,\, m_d} d^{-n} \left( \frac{\lambda}{1+\lambda}
  \right)^{\sum\limits_{i=1}^d \frac{m_i-k_i}{2}} \left( \frac{1}{1+\lambda}
       \right)^{\sum\limits_{i=1}^d \frac{m_i + k_i}{2}} \prod_{m_j>0} P_{m_j,\, k_j}(0), \\
    =\nonumber& \lambda^{\frac{n - \sum_{i=1}^d k_i}{2}} \sum_m {n \choose m_1,\, \cdots,\, m_d} \left(\frac{1}{d(1+\lambda)}\right)^n  \prod_{m_j>0} P_{m_j,\, k_j}(0), 
  \end{align}
  where the sum is over all tuples $m = (m_1,\, \cdots,\, m_d) \in \Z^d$ such that
$|m| = n$ and $m_j \geq |k_j|$ for $1 \leq j \leq d$.

Note that
\[
  \E \left[ e^{\sum_{i=1}^d s_i Z_n^i} I_{\{Z_n \in \Z_+^d\}} \,\big|\, Z_0 =
    \zero \right]
  \geq
  \left( e^{s_1} \wedge 1 \right) \E \left[ e^{\sum_{i \in I_2} s_i Z_n^i} I_{\{
    Z_n^{I_1} = y_0,\, Z_n^{I_2} \in \Z_+^{I_2} \}} \,\big|\, Z_0 = \zero \right]. \]
  For $\epsilon = \left( \epsilon_i \right)_{i \in I_2} \in \left\{ -1,\, 1
  \right\}^{I_2}$, let 
  \[ \sO_{\epsilon} = \left\{ y = \left( y_i \right)_{i \in I_2} \in \Z^{I_2} :\
      \epsilon_i y_i \geq 0,\ i \in I_2 \right\}. \]
Note that for every $k \in \Z^d$, $P_{m_j,\, k_j}(0) = P_{m_j,\,
  \epsilon_j k_j}(0)$. Therefore, by (\ref{orthant1})
\[ \P \left( Z_n = k \,\big|\, Z_0 = \zero \right)
  =
  \lambda^{\sum_{i=1}^d \frac{1}{2} \left( \epsilon_i - 1 \right) k_i}
  \P \left( Z_n = \left( \epsilon_1 k_1,\, \cdots,\, \epsilon_d k_d \right)
    \,\big|\, Z_0 = 0 \right).
\]
Applying the fact that $e^{s_i} \lambda^{-1/2} \geq 1$ for $i \in I_2$, we
obtain that for every $\epsilon \in \left\{ -1,\, 1 \right\}^{I_2}$, 
\begin{align*}
  & \E \left[ e^{\sum_{i \in I_2} s_i Z_n^i} I_{\{
    Z_n^{I_1} = y_0,\, Z_n^{I_2} \in \Z_+^{I_2} \}} \,\big|\, Z_0 = \zero \right] \\
    \geq&
        \sum_{k \in \Z_+^{I_2}} 
        \left( e^{s_i} \lambda^{-1/2} \right)^{\sum_{i \in I_2} \left( 1 - \epsilon_i \right) k_i}e^{\sum_{i\in I_2} s_i \epsilon_i k_i}  \P \left( Z_n^{I_1} = y_0,\, Z_n^{I_2} = \left( \epsilon_i k_i\right)_{i\in I_2} \,\big|\, Z_0 = \zero \right) \\
  \geq &
         \E \left[ e^{\sum_{i\in I_2} s_i Z_n^i} I_{\{ Z_n^{I_1} = y_0, Z_n^{I_2} \in \sO_{\epsilon}\}} \,\big|\, Z_0 = \zero \right]. 
\end{align*}
By taking sum over all $\epsilon \in \{-1,\, 1\}^{I_2}$, we complete the proof of
this lemma. 
\end{proof}

Recall the definition of $\psi$ from \eqref{psi}.

\begin{lem}
  \label{L:Lambdau}
  For every $s \in \R^d$ and $x \in \Z^d$ we have that
  \[ \limsup_{n \to \infty} \frac{1}{n} \Lambda_n(s,\, x) \leq \ln \psi(s). \]
\end{lem}

\begin{proof}
  By Lemma~\ref{L:Lambda0}, we only need to consider the case $x = \zero$. 
  Recall that $s_0:= \frac{1}{2} \ln \lambda$. 
  For $s \in \R^d$, let $\tilde{s} := \left( s_1\vee s_0,\, \cdots, \, s_d
    \vee s_0 \right)$. Since $\Lambda_n(s,\, 0)$ is increasing in each
  coordinate $s_i$, we have that $\Lambda_n(s,\, 0) \leq \Lambda_n(\tilde{s},\,
  0)$. Let $(Z_n)$ be the drifted random walk defined by \eqref{e:Zn}. Then, by
  Lemma~\ref{L:hkupper},
  \begin{align*}
    & \E_{\zero} \left[ e^{\sum_{i=1}^d \tilde{s}_i |X_n^i|} \right] \leq 2^d  \sum_{k \in \Z_+^d} e^{\sum_{i=1}^d \tilde{s}_i k_i} 
    \P \left( Z_n = k \,\big|\, Z_0 = \zero \right) \leq 2^d \E \left[
      e^{\sum_{i=1}^d \tilde{s}_i Z_n^i} \,\big|\, Z_0 = \zero \right] \\
    =& 2^d
    \left( \sum_{i=1}^d \frac{\lambda e^{-\tilde{s}_i} +
        e^{\tilde{s}_i}}{d(1+\lambda)} \right)^n = 2^d \left( \psi(s) \right)^n.
  \end{align*}
  The lemma follows immediately. 
\end{proof}

\begin{lem}
  \label{L:Lambdal}
  For every $s \in \R^d$ and $x \in \Z^d$ we have that
  \[ \liminf_{n \to \infty} \frac{1}{n} \Lambda_n(s,\, x) \geq \ln \psi(s). \]
\end{lem}

\begin{proof}
  It suffices to prove only for $n$ even. 
  We use the same notations as in Lemma~\ref{L:orthant}. Fix $x \in \Z_+^d
  \setminus \sX$. By Lemma~\ref{L:hklower} and~\ref{L:orthant}, we have
  \begin{align*}
    \E_x \left[ e^{\sum_{i=1}^d s_i |X_{2n}^i|}  \right]
    \geq&
    n^{-d} \sum_{k \in \Z_+^d} e^{\sum_{i=1}^d s_i k_i} \P \left( Z_{2n} = k \,\big|\, Z_0 = x \right) \\
    =& n^{-d} \E \left[ e^{\sum_{i=1}^d s_i Z_{2n}^i} I_{\{Z_{2n} \in \Z_+^d\}} \,\big|\, Z_0 =x \right] \\
    \geq &
           c n^{-d} \E \left[ e^{\sum_{i \in I_2} s_i Z_{2n}^i} I_{\{Z_{2n}^{I_1} = x^{I_1} \}} \,\big|\, Z_0 = x \right] 
  \end{align*}
  for some constant $c > 0$. Let $\Gamma_1(2n)$  be the
  number of nearest-neighbor paths of length $2n$ in $\Z^{I_1}$ from $\zero^{I_1}$ to
  $\zero^{I_1}$. Similarly, for $k \in \Z^{I_2}$, let $\Gamma_2(2n,\, k)$ be
  the number of paths of 
  length $2n$ in $\Z^{I_2}$ from $\zero^{I_2}$ to $k$. Then we have for $k \in
  \Z^{I_2}$,
  \[
    \P \left( Z_{2n}^{I_1} = x^{I_1}, \, Z_{2n}^{I_2} = x^{I_2} + k \,\big|\,
      Z_0 = x \right)
    =
    \sum_{m=0}^{n} {2n \choose 2m} \frac{\lambda^{n - \sum_{i \in I_2}
        k_i}}{\left[ d \left( 1+\lambda \right) \right]^{2n}} \Gamma_1(2m)
    \Gamma_2(2n-2m,\, k).
  \]
  Recall that $N(s) = |I_1|$. 
  Let $(W_n)$ be the drifted random walk in $\Z^{I_2}$ starting at $\zero$, that
  is, the transition 
  probability is given by \eqref{e:Zn} with $d$ replaced by $d - N(s)$. Since
  $\Gamma_1(2m) \geq c m^{-N(s)/2} \left( 2 N(s) \right)^{2m}$, we obtain that 
  \begin{align*}
   & \P \left( Z_{2n}^{I_1} = x^{I_1}, \, Z_{2n}^{I_2} = x^{I_2} + k \,\big|\,
      Z_0 = x \right) \\
    \geq& 
c n^{-N(s)/2} d^{-2n} \sum_{m=0}^{n} {2n \choose 2m} \left( N(s) \rho_{\lambda} \right)^{2m} 
 \left( d - N(s) \right)^{2n-2m} \P \left( W_{2n-2m} = k \right). 
  \end{align*}
  Therefore, 
  \begin{align*}
    & \E_x \left[ e^{\sum_{i=1}^d s_i |X_{2n}^i|}  \right] \\
    \geq&
     n^{-3d/2} d^{-2n}\sum_{m=0}^n \sum_{k \in \Z^{I_2}} e^{\sum_{i \in I_2} s_i (k_i+ x_i)} {2n \choose 2m} \left(
      N(s) \rho_{\lambda} \right)^{2m} \left( d-N(s) \right)^{2n-2m} \P \left(
          W_{2n-2m} = k \right) \\
    = & c n^{-3d/2} d^{-2n}\sum_{m=0}^n {2n \choose 2m} \left(
        N(s) \rho_{\lambda} \right)^{2m} \left( d-N(s) \right)^{2n-2m} \E \left[ e^{\sum_{i \in I_2} s_i W_{2n-2m}^i} \right] \\
    = & c n^{-3d/2} d^{-2n} \sum_{m=0}^n {2n \choose 2m} \left(
        N(s) \rho_{\lambda} \right)^{2m} \left( d-N(s) \right)^{2n-2m} \left( \frac{\sum_{i \in I_2} \left( \lambda e^{-s_i} + e^{s_i} \right)}{(d-N(s)) (1+\lambda)} \right)^{2n-2m} \\
    \geq & \frac{c}{3} n^{-3d/2} \left( \psi(s) \right)^{2n}. 
  \end{align*}
  The last inequality holds since for any positive real number $a, b,$
  \[\frac{\sum\limits_{m=0}^{n} \binom{2n}{2m} a^{2n-2m} b^{2m}}{\left(a+ b\right)^{2n}}\rightarrow \frac{1}{2}, \ \ n\rightarrow \infty.
  \]
  The proof is finished. 
 \end{proof}

Combining Lemma~\ref{L:Lambdau} and~\ref{L:Lambdal}, we get the following
result.
\begin{cor}
\label{C:limit}
  For every $s \in \R^d$ and $x \in \Z^d$, we have
  \begin{equation*}
    \label{e:limit}
    \lim_{n \to \infty} \frac{1}{n} \Lambda_n(s,\, x) = \ln \psi(s). 
  \end{equation*}
\end{cor}

Define $\Lambda(s) := \ln \psi(s)$ and let $\Lambda^{\ast}(x) = \sup\limits_{s \in \R^d}
\left\{ \left( s,\, x \right) - \Lambda(s) \right\}$ be Fenchel-Legendre
transform of $\Lambda$. 

\begin{lem}\label{property0-1} Give any $d\geq 1$ and $\lambda\in (0,1).$ The effective domain of $\Lambda^{\ast}(\cdot)$ is
$$\mathcal{D}_{\Lambda^{\ast}}(\lambda):=\left\{x\in\mathbb{R}^d:\ \Lambda^\ast(x)<\infty\right\}=\left\{x=(x_1,\cdots,x_d)\in\mathbb{R}^d:\ \sum\limits_{i=1}^dx_i\leq 1,\ 0\leq x_i\leq 1,\ 1\leq i\leq d\right\}.$$
Furthermore, $\Lambda^{\ast}(\cdot)$ is strictly convex in
$x=(x_1,\cdots,x_d)\in \mathcal{D}_{\Lambda^{\ast}}(\lambda)$ with
$\sum\limits_{i=1}^dx_i<1,$ and $\Lambda^{\ast}(x) = 0$ if and only if $x =
\left(\frac{1-\lambda}{d(1+\lambda)},\cdots,\frac{1-\lambda}{d(1+\lambda)}\right)$. 
\end{lem}
\pf
Note that $s_0=\frac{1}{2}\ln\lambda.$ 
%
%
%
Firstly, let $$\mathcal{D}=\left\{x=(x_1,\cdots,x_d)\in\mathbb{R}^d: 0\leq x_i\leq \sum\limits_{j=1}^{d}x_j\leq 1,\ 1\leq i\leq d\right\}.$$
Then for any $x\in \mathcal{D}$, we have
\begin{eqnarray}\label{rate1}
\nonumber \Lambda^*(x)&=&\sup\limits_{s_i\geq s_0,1\leq i\leq d}\ln\left\{\frac{1}{2}(1+\lambda)\frac{e^{\sum\limits_{i=1}^ds_ix_i}}{\frac{1}{2d}\sum\limits_{i=1}^d(\lambda e^{-s_i}+e^{s_i})}\right\}\\
   \nonumber&=&\sup\limits_{y_i\geq 0,1\leq i\leq d}\ln\left\{\frac{1}{2}(1+\lambda)\lambda^{-\frac{1}{2}+\frac{1}{2}\sum\limits_{i=1}^dx_i}\frac{e^{\sum\limits_{i=1}^dy_ix_i}}{\frac{1}{2d}\sum\limits_{i=1}^d(e^{-y_i}+e^{y_i})}\right\}\\
   &=&\frac{1}{2}\sum\limits_{i=1}^d x_i \ln (\lambda)- \ln(\rho_{\lambda})+\sup\limits_{y_i\geq 0,1\leq i\leq d}\ln\left\{\frac{e^{\sum\limits_{i=1}^dy_ix_i}}{\frac{1}{2d}\sum\limits_{i=1}^d(e^{-y_i}+e^{y_i})}\right\}< \infty.
\end{eqnarray}
In fact, for any $s=(s_1,\cdots,s_d)\in\mathbb{R}^d,$ $\sum\limits_{i=1}^ds_ix_i$ is increasing in each $s_i$ and
$\psi(s)=\psi\left(\tilde{s}\right)$, where $\tilde{s}$ is defined in Lemma \ref{L:Lambdau}. And for $x \in \mathcal{D}^c,$ it's easy to verify that $\Lambda^*(x)= \infty.$ 

By \cite[Lemma~2.3.9]{DZ1998}, $\Lambda^*$ is a good convex rate function. It's obvious that the Hessian matrix of $\Lambda(s)$ is positive-definite which implies strict concavity of $s\cdot x- \Lambda(s)$, thus the local maximum of $s\cdot x- \Lambda(s)$ exists uniquely and is attained at a finite solution $s=s(x)$, i.e.
\begin{eqnarray*}\label{solution}
\Lambda^{\ast}(x)=s(x)\cdot x- \Lambda(s(x)).
\end{eqnarray*}
Then according to implicit function theorem, $\Lambda^{\ast}(\cdot)$ in $\left\{x=(x_1,\cdots,x_d)\in\mathbb{R}_+^d:\ \sum\limits_{i=1}^{d} x_i< 1\right\}$ is strictly convex.

Finally, due to the strict convexity of $\Lambda^{*}(\cdot)$ and (\ref{rate1}), $
\left(\frac{1-\lambda}{d(1+\lambda)},\cdots,\frac{1-\lambda}{d(1+\lambda)}\right)$ is the unique solution of $\Lambda^{*}(x)= 0$.
\qed
\vskip 3mm

Assume $d=1, 2 \mbox{ and } \lambda\in (0,1)$, we can get explicit formula of $\Lambda^{*}$ by calculating rate function $\overline{\Lambda^{*}}$ of SRW which is omitted here. For $\lambda=0$, we have the following explicit expression.\\

\begin{lem}\label{property0}
For any $d\geq 2$ and $\lambda=0,$
\begin{eqnarray*}
&&\mathcal{D}_{\Lambda^{\ast}}(0):=\left\{x\in\mathbb{R}^d:\ \Lambda^\ast(x)<\infty\right\}=\left\{x=(x_1,\cdots,x_d)\in\mathbb{R}^d_+:\ \sum\limits_{i=1}^dx_i=1\right\},\\
&&\Lambda^\ast (x)=\left\{\begin{array}{cl}
  \ln d+\sum\limits_{i=1}^dx_i\ln x_i,\ &x\in\mathcal{D}_{\Lambda^\ast}(0),\\
  \ \\
  +\infty, \ & otherwise,
  \end{array}
\right.\\
&&\left\{x\in\mathbb{R}^d:\ \Lambda^\ast(x)=0\right\}=\left\{\left(\frac{1}{d},\cdots,\frac{1}{d}\right)\right\}.
\end{eqnarray*}
\end{lem}

\pf  Clearly 
$$\mathcal{D}_{\Lambda^{\ast}}(0)\subseteq \left\{x=(x_1,\cdots,x_d)\in\mathbb{R}^d_+:\ \sum\limits_{i=1}^dx_i=1\right\}.$$
Assume firstly $x_i> 0, 1\leq i\leq d.$ Let $y_i=e^{s_i},\ 1\leq i\leq d,$ then by the Jensen inequality,
\begin{eqnarray*}
&&s\cdot x-\ln\left(\frac{1}{d}\sum\limits_{i=1}^de^{s_i}\right)=\sum\limits_{i=1}^dx_i\ln\frac{y_i}{x_i}-\ln\left(\sum\limits_{i=1}^dy_i\right)+ \sum\limits_{i=1}^dx_i\ln x_i+\ln d \\
&&\ \ \ \ \leq \ln\left(\sum\limits_{i=1}^d x_i\frac{y_i}{x_i}\right)-\ln\left(\sum\limits_{i=1}^dy_i\right)+\sum\limits_{i=1}^dx_i\ln x_i+\ln d\\
&&\ \ \ \ =\sum\limits_{i=1}^dx_i\ln x_i+\ln d,
\end{eqnarray*}
and the inequality in the second line becomes equality only if each $s_i=\ln x_i.$ 

If there exists some $i$ such that $x_i=0,$ we still have that   
\begin{eqnarray*}
s\cdot x-\ln\left(\frac{1}{d}\sum\limits_{i=1}^de^{s_i}\right)\leq \sum\limits_{i=1}^dx_i\ln x_i+\ln d,
\end{eqnarray*}
then $\sup\limits_{s\in\mathbb{R}^d}\left\{s\cdot x-\ln\left(\frac{1}{d}\sum\limits_{i=1}^de^{s_i}\right)\right\}=\sum\limits_{i=1}^dx_i\ln x_i+\ln d$ by lower semi-continuity of $\Lambda^*(\cdot).$
Hence 
\begin{eqnarray*}
&&\mathcal{D}_{\Lambda^{\ast}}(0)=\left\{x=(x_1,\cdots,x_d)\in\mathbb{R}^d_+:\ \sum\limits_{i=1}^dx_i=1\right\},\\
&&\left\{x\in\mathbb{R}^d:\ \Lambda^\ast(x)=0\right\}=\left\{\left(\frac{1}{d},\cdots,\frac{1}{d}\right)\right\}.
\end{eqnarray*}
\qed
\vskip 3mm

Denote by $\mathbb{F}$ the set of exposed points of $\Lambda^{\ast}(\cdot)$ whose exposing hyperplane belong to $\mathcal{D}_{\Lambda}^{o}$.
Here $y\in\mathbb{R}^d$ is an exposed point of $\Lambda^\ast$ if for some $s\in\mathbb{R}^d,$
$$(y,s)-\Lambda^\ast(y)>(x,s)-\Lambda^\ast(x),\ \forall x\in\mathbb{R}^d\setminus\{y\};$$
and we call the above $s$ an exposing hyperplane.
\begin{lem}\label{infimumattained}
For any open set $G$ of $\mathbb{R}^d,$
\[\inf_{x\in G\cap{\mathbb{F}}} \Lambda^{\ast}(x)=\inf_{x\in G} \Lambda^{\ast}(x)\].
\end{lem}
\pf Assume $\lambda\in (0,1)$ and $d\geq 1.$  
By Duality lemma \cite[Lemma~4.5.8]{DZ1998}, we have that
$$\left\{x=(x_1,\cdots,x_d)\in\mathbb{R}_+^d:\ \sum\limits_{i=1}^d x_i< 1\right\}\subseteq \mathbb{F}.$$
By strict convexity of $\Lambda^*(\cdot)$ and (\ref{rate1}), for any $x=(x_1,\cdots,x_d)\in\mathbb{R}_+^d$ with $\sum\limits_{i=1}^dx_i=1,$
$$\Lambda^*(x)\geq\limsup\limits_{n\rightarrow\infty}\Lambda^*\left[x-\frac{1}{n}\left(x-\left(\frac{1-\lambda}{d(1+\lambda)},\cdots, \frac{1-\lambda}{d(1+\lambda)}\right)\right)\right]$$
which leads to the conclusion.
\vskip 2mm

Assume $\lambda=0$ and $d\geq 2.$ We have that
$$\mathbb{F}= \left\{x=(x_1,\cdots,x_d)\in\mathbb{R}_+^d:\ \sum\limits_{i=1}^d x_i = 1,\mbox{and}\ x_i >0, 1\leq i\leq d\right\}.$$
Then by Lemma \ref{property0}, for any open set $G,$
\[\inf_{x\in G\cap \mathbb{F}} \Lambda^{\ast}(x)=\inf_{x\in G} \Lambda^{\ast}(x).\]
\qed
\vskip 3mm

\noindent{\bf Proof of Theorem \ref{thm3.1}.} Let $\mu_n$ be the law of $\left(\frac{\left\vert X_n^1\right\vert}{n},\cdots,\frac{\left\vert X_n^d\right\vert}{n}\right)$ for any $n\in\mathbb{N}.$ From
Corollary~\ref{C:limit}, Lemma~\ref{property0-1} and \ref{property0}, the logarithmic moment generating function exists with $\Lambda(s)= \ln \psi(s)$ and
\[\mathcal{D}_\Lambda=\left\{s\in\mathbb{R}^d:\ \Lambda(s)<\infty\right\}=\mathbb{R}^d=\mathcal{D}_\Lambda^o.\]
By $\bf{(a)}$ and $\bf{(b)}$ of the Gartner-Ellis theorem and Lemma \ref{infimumattained}, we have that
for any closed set $F\subseteq\mathbb{R}^d,$
\[\limsup\limits_{n \rightarrow \infty} \frac{1}{n} \ln\mu_n(F)\leq-\inf\limits_{x\in F}\Lambda^{\ast}(x);\]
and for any open set $G$ of $\mathbb{R}^d,$
\[\liminf\limits_{n \rightarrow \infty} \frac{1}{n} \ln \mu_n(G)\geq-\inf\limits_{x\in G} \Lambda^{\ast}(x).\]
The proof is done.
\qed

\bigskip
\noindent{\bf Acknowledgements.} Y. Liu, L. Wang and  K. Xiang thank NYU Shanghai - ECNU Mathematical Institute for hospitality and financial support. V. Sidoravicius thanks Chern Institute of Mathematics for
hospitality and financial support.

\end{document}